\documentclass[twoside,11pt]{article}

\usepackage{amsmath, mathtools, amssymb, amsthm}
\usepackage{natbib}
\usepackage{microtype}
\usepackage{enumerate, multicol}
\usepackage{bbm}
\usepackage[linesnumbered,ruled,vlined,algo2e]{algorithm2e}
\usepackage[font=small,labelfont=bf]{caption}
\usepackage{array} 
\usepackage[table]{xcolor}
\usepackage{tikz}
\usetikzlibrary{positioning, arrows, shapes.geometric, calc}
\usepackage{fullpage}
\usepackage[colorlinks=true, linkcolor=blue, urlcolor=blue, citecolor=blue]{hyperref}
\usepackage{algorithm,algorithmic}
\usepackage{breakcites}


\definecolor{gaussianfill}{RGB}{220,230,250}
\definecolor{kssgfill}{RGB}{220,250,230}
\definecolor{usgfill}{RGB}{245,230,250}

\SetKwInput{KwInput}{Input}
\SetKwInput{KwOutput}{Output}
\SetKwRepeat{KwRepeat}{repeat}{until}

\theoremstyle{plain}
\newtheorem{theorem}{Theorem}[section]
\newtheorem{lemma}[theorem]{Lemma}

\newtheorem{remark}[theorem]{Remark}
\newtheorem{definition}[theorem]{Definition}
\numberwithin{equation}{section}

\newcommand{\EE}{\mathbb{E}}

\newcommand{\NN}{\mathbb{N}}

\newcommand{\RR}{\mathbb{R}}

\newcommand{\diam}{\mathrm{diam}}

\newcommand{\sign}{\operatorname{sign}}

\newcommand{\argmin}{\mathop{\mathrm{argmin}}}

\def\T{^\top}

\allowdisplaybreaks

\begin{document}

\begin{center}
{\LARGE Fast Near-Optimal Estimation over Symmetric Norm Balls}

{\large
\begin{center}
Matey Neykov
\end{center}}

{Department of Statistics and Data Science, Northwestern University\\
\texttt{mneykov@northwestern.edu}}
\end{center}

\begin{abstract}
This short note proposes a polynomial-time algorithm for near-optimal Euclidean estimation of a signal constrained to lie in the unit ball of a symmetric norm, where the symmetry is with respect to a known basis and the norm is accessible through an evaluation oracle. We further extend the method to a random-design, moderate-dimensional linear regression setting, where the regression parameter is likewise assumed to belong to a constraint set defined by a symmetric norm.
\end{abstract}

\tableofcontents
\section{Introduction}

We consider the Gaussian sequence model (GSM)
\begin{align*}
    Y = \mu + \xi,
\end{align*}
with $\xi \in \RR^n$ being a vector comprised of independent sub-Gaussian entries with common sub-Gaussian parameter $\sigma$\footnote{i.e., $\EE \exp(\lambda \cdot \xi_i) \leq \exp(\lambda^2 \sigma^2/2)$, for all $\lambda \in \RR$. } (assumed to be known), where the unknown mean vector \(\mu\in\mathbb R^n\) is known to belong to a
constraint set \(K\). This model is one of the basic testing grounds for
understanding the interaction between geometry, computation, and statistical
optimality.  It is simple enough to admit sharp theory, yet rich enough to
capture many of the phenomena that arise in nonparametric estimation,
high-dimensional statistics, and signal denoising.  In particular, when \(K\)
is a convex body, the maximum likelihood estimator is simply the Euclidean
projection of \(Y\) onto \(K\).  A central question is therefore: when is this natural estimator minimax optimal, and when it is not, can one construct a computationally tractable replacement?

In this paper we assume further that the Minkowski functional of $K$ is a symmetric norm.  Throughout, \(K\) is the unit ball of a symmetric norm on
\(\mathbb R^n\), meaning that
\begin{align*}
    \|(x_1,\ldots,x_n)\|_K
    =
    \|(\eta_1 x_{\pi(1)},\ldots,\eta_n x_{\pi(n)})\|_K
\end{align*}
for every permutation \(\pi\) of \([n] = \{1,\ldots, n\}\) and every choice of signs
\(\eta_i\in\{-1,1\}\).  Of course, the theory also remains valid for any known rotation of such a constraint. This class contains all \(\ell_p\) balls, \(1\le p\le\infty\), as well as many other natural orthosymmetric and
permutation-invariant constraints.  The key advantage of this class is that
its entropy structure is well understood, as we will explain soon.

The information-theoretic side of our GSM estimation question is now fairly well understood. For bounded convex constraints, \cite{neykov2022minimax} showed that the minimax risk under squared Euclidean loss is characterized, up to universal constants, by a local entropy fixed point.  More precisely, if \(\log M^{\operatorname{loc}}(\eta)\) denotes the local metric entropy
of \(K\) (see Definition \ref{definition:local:metric:entropy}), then the critical radius
\begin{align}\label{minimax:rate}
        \eta^* = \sup\left\{ \eta > 0: \frac{\eta^2}{\sigma^2} \le \log M^{\operatorname{loc}}(\eta)
    \right\},
\end{align}
determines the minimax rate (in the case of Gaussian noise), up to the diameter of the set.  This local entropy perspective was subsequently extended beyond convexity to star-shaped constraints in \cite{prasadan2026information}.  These results give a sharp answer to the statistical question, but they do not by themselves produce efficient algorithms for general constraint sets.

A different line of recent work asks when computationally natural estimators
are already optimal.  The least squares estimator is minimax optimal for many classical convex constraints, but it is not universally optimal.  In
\cite{prasadan2024some}, the risk of the LSE in the Gaussian sequence model was characterized through the behavior of local Gaussian widths, yielding necessary and sufficient conditions for minimax optimality and exhibiting several examples where the LSE is suboptimal.  In particular, it was argued that for \(\ell_p\) balls with \(1<p<2\), the LSE is suboptimal for suitable noise levels. Related phenomena were later studied in \cite{aolaritei2025revisiting}, who showed that the LSE over \(\ell_p\) balls can be rate-suboptimal in the range \(1<p<2\) for broad regimes of the parameters. These observations motivate the search for estimators that are both computationally simple and minimax near-optimal for symmetric norm type constraints. We note here that the case of $\ell_p$ balls for $1 \leq p < 2$ admits a solution via soft or hard thresholding; see, for example, the monograph of \cite{johnstone2011gaussian}. However, it is unclear whether these estimators remain near minimax optimal for any symmetric norm ball.

Recently, \cite{neykov2025polynomial} proposed polynomial-time
near-optimal estimators for certain origin-symmetric type-2 convex bodies,
under appropriate oracle access and regularity assumptions.  This provides a
general algorithmic framework for a large class of convex constraints.
However, the class of type-2 bodies does not uniformly cover the \(\ell_p\)
balls for \(1\le p<2\), which are among the most important examples in
sequence-space estimation.  Thus, even for the classical family of
\(\ell_p\) balls, there remains a gap between the information-theoretic
minimax theory and a general, easy-to-implement, near-optimal estimator.


In this work we propose a near-optimal polynomial time algorithm for any (well-balanced)\footnote{Here and throughout by well-balanced set $K$ we mean there exist known $r, R \in \RR_+$ so that $rB_2\subseteq K \subseteq RB_2$. We will see later that in fact we can compute such $(r,R)$ for any symmetric norm ball given only oracle access to its Minkowski functional.} symmetric norm ball provided that we have access to an oracle evaluating its Minkowski functional. The reason why our construction is possible and natural is the central result of \cite{edmunds1998entropy} which identifies the entropy numbers of the set $K$ in $\ell_2$ norm,
in terms of a soft-thresholding functional of the decreasing rearrangement of vectors in \(K\).  This result suggests that, for symmetric norm balls, the local entropy fixed point appearing in the minimax theory should be accessible through sparse approximation.

To state the relevant quantity, let \(a^*=(a_1^*,\ldots,a_n^*)\) denote the
nonincreasing rearrangement of \((|a_1|,\ldots,|a_n|)\).  For \(1\le s\le n\),
define the soft-tail functional
\begin{align*}
    u_s(K)
    :=
    \sup_{a\in K}
    \left(
        \sum_{j=1}^n \min\{a_j^*,a_s^*\}^2
    \right)^{1/2}.
\end{align*}
The theorem of Edmunds and Netrusov shows, in the relevant intermediate
regime, that the dyadic entropy numbers satisfy
\begin{align*}
    e_k(K) \asymp u_s(K),
    \qquad
    s = \lceil k/\log(1 + n/k)\rceil,
\end{align*}
for $k < n/2$. As the minimax rate of the problem is given by an entropic equation, and the local entropy of $K$ is intimately related to its entropy numbers, an estimator based on thresholding becomes very natural. We now describe our estimator (in an ideal case where we can compute the projections below in polynomial time).

Our estimator is a penalized sparse projection estimator.  For each sparsity
level \(s\), let \(T_s(Y)\) be the set of indices corresponding to the \(s\)
largest coordinates of \(Y\) in absolute value, and define
\begin{align*}
    \widehat\mu_s
    \in
    \arg\min_{\theta\in K,\ \operatorname{supp}(\theta)\subseteq T_s(Y)}
    \|Y-\theta\|_2^2 .
\end{align*}
We then choose the sparsity level by the penalized criterion
\begin{align*}
    \widehat s
    \in
    \arg\min_{0\le s\le n}
    \left\{
        \|Y-\widehat\mu_s\|_2^2
        +
        \lambda\sigma^2 s\log(en/s)
    \right\},
\end{align*}
where \(\lambda>0\) is a sufficiently large universal constant and set
\begin{align}\label{estimator:def:introduction}
        \widehat\mu=\widehat\mu_{\widehat s},
\end{align}
and with the convention that $0\cdot \log (en/0) = 0$.

The procedure is deliberately simple: it combines sorting, sparse projection, and a complexity penalty for the choice of support size.  For standard examples such as \(\ell_p\) balls, sparse projections are explicit or computationally inexpensive, and the estimator is clearly related to familiar hard thresholding-type procedures.

Our main result shows that this estimator is minimax near-optimal over every
symmetric norm ball.  

For \(\ell_p\) balls, \(1\le p<2\), the resulting estimator recovers the
classical nonlinear behavior known from the theory of Gaussian sequence and
wavelet estimation; see, for example, the monograph of
\cite{johnstone2011gaussian}.  Thus, the present work can be viewed as a
generalization of classical hard and soft thresholding ideas from \(\ell_p\)
balls to arbitrary symmetric norm balls.  At the same time, it gives an
algorithmic realization of the local entropy minimax theory: the estimator
does not require the construction of entropy packings or the solution of a
general nonconvex optimization problem, but instead exploits the symmetry of
the norm to reduce the problem to sparse approximation and penalized model
selection.

The remainder of the paper is organized as follows. Section~\ref{main:section} contains our estimator, main result, and its proof. Section~\ref{regression:section} extends the idea to a linear regression with moderate dimension. Finally, a brief discussion is given in Section~\ref{discussion:section}

\section{Estimator and Main Result}\label{main:section}
This section presents our main result. Before we begin, we give several useful definitions. 
\subsection{Notation and Definitions}
We will use the shorthand notation $d := \diam(K)$. For an integer $n\in \NN$ we use the convenient notation $[n] = \{1,\ldots,n\}$.  We use $\|\cdot\|_{\operatorname{op}}$ and $\|\cdot\|_F$ to denote the operator and Frobenius norms of a matrix respectively. We will denote the $n$-dimensional Euclidean unit ball with $B_2$. For a vector $x \in \RR^n$ we denote its support (i.e. collection of its non-zero entries) by $\operatorname{supp}(x)$. We denote the Minkowski functional (gauge) of an origin-symmetric convex body $K$ by $\|x\|_K = \inf\{t > 0 : x \in t K\}$. Note that $\|\cdot\|_K$ is a norm if the set $K$ is origin-symmetric. The first formal definition below is that of a symmetric norm. After the definition we give common examples of symmetric norms.
\begin{definition}[Symmetric norm]\label{symmetric:norm:def} A norm $\|\cdot\|_K$ is symmetric if for any $a \in \RR^n$ we have
\begin{align*}
    \bigg \|\sum_{i = 1}^n a_i e_i\bigg\|_K = \bigg\|\sum_{i = 1}^n \eta_i a_{\pi(i)} e_i\bigg\|_K,
\end{align*}
for any permutation $\pi$, and signs $\eta_i \in \{-1, 1\}$, where $\{e_i\}_{i \in [n]}$ denotes an orthonormal basis in $\RR^n$.
\end{definition}
Since we will be assuming that the basis $\{e_i\}_{i \in [n]}$ is known in this work, without loss of generality and as we mentioned in the introduction we will henceforth assume that $\{e_i\}_{i \in [n]}$ is the standard basis. 

\begin{remark}
Many commonly used constraints are unit balls of symmetric norms.  The most
basic examples are the $\ell_p$ norms, $1\le p\le \infty$, whose unit balls are
invariant under arbitrary permutations and sign changes of the coordinates.
Beyond the standard $\ell_p$ family, several other important constraints also
fall into this class.  For example, convex weak $\ell_p$ balls, as considered
in \cite{neykov2022minimax}, are symmetric norm balls, since their definition
depends only on the decreasing rearrangement of the coordinate magnitudes.
Similarly, the SLOPE penalty \cite{bogdan2015slope},
\[
    \|x\|_{\operatorname{SLOPE}}
    =
    \sum_{j=1}^n \lambda_j x_j^*,
    \qquad
    \lambda_1\ge \lambda_2\ge \cdots \ge \lambda_n\ge 0,
\]
where \(x_1^*\ge \cdots \ge x_n^*\) denotes the decreasing rearrangement of
\((|x_1|,\ldots,|x_n|)\), defines a symmetric norm.  This class also includes norms obtained by combining symmetric norms, such as
\[
    x \mapsto \alpha \|x\|_1+\beta \|x\|_2,
    \qquad \alpha,\beta\ge 0.
\]
Other examples include top-$k$ norms of the form
\[
    \|x\|_{(k)}=\sum_{j=1}^k x_j^*,
\]
and more general Lorentz-type norms, whose values are determined by the
ordered magnitudes of the coordinates.  Thus the class of symmetric norm balls is considerably richer than the classical $\ell_p$ examples, while still retaining enough structure to allow for efficient estimation.
\end{remark}

The following definition of local metric entropy plays a fundamental role in the determining the minimax rate. 

\begin{definition}[Local Metric Entropy] \label{definition:local:metric:entropy} Given a set $K$, define the $\eta$-packing number of $K$ as the maximum cardinality $M=M(\eta,K)$ of a set $\{\nu_1,\dots,\nu_{M}\}\subset K$ such that $\|\nu_i-\nu_j\|_2 > \eta$ for all $i\ne j$. Given a fixed sufficiently large constant $c > 0$, define the local metric entropy of $K$ as  $\log M_K^{\operatorname{loc}}(\eta)$ where $M_K^{\operatorname{loc}}(\eta) = \sup_{\nu\in K} M(\eta/c, (\nu + B_2\eta)\cap K)$. When clear from the context we drop the index $K$ from $M_K^{\operatorname{loc}}(\eta)$.
\end{definition}

As shown in \cite{neykov2022minimax} the local entropy is a non-increasing function in $\eta$. Next, we define the dyadic entropy numbers of $K$.

\begin{definition}[Dyadic Entropy Numbers of a Convex Body] The $k$-th dyadic entropy number of $K$ is given by
\begin{align*}
e_k(K)
:=
\inf\left\{
\varepsilon>0:
N(\varepsilon, K)\le 2^{k-1}
\right\},
\qquad k\ge 1,
\end{align*}
where \(N(\varepsilon, K)\) denotes the smallest number of Euclidean balls of radius \(\varepsilon\) needed to cover \(K\).
\end{definition}

We end this section by showing that any symmetric norm ball $K$ is ``well-balanced'' in the sense that we can compute $(r, R)$ so that $r B_2 \subseteq K \subseteq RB_2$ given oracle access to its Minkowski functional.

\begin{lemma}[Euclidean sandwich for symmetric norm balls]\label{well-balancedness-norm}
Let $\|\cdot\|_K$ be a norm on $\mathbb{R}^n$ whose unit ball
\[
K := \{x\in \mathbb{R}^n : \|x\|_K \le 1\}
\]
is sign- and permutation-invariant (i.e. it's symmetric). Suppose that we have oracle access to
$\|\cdot\|_K$, and let
\[
a := \|e_1\|_K .
\]
Then
\[
\frac{1}{a\sqrt n} B_2^n \subseteq K \subseteq \frac{\sqrt n}{a} B_2^n .
\]
In particular, one may take
\[
r = \frac{1}{a\sqrt n},
\qquad
R = \frac{\sqrt n}{a}.
\]
\end{lemma}

The proof is relegated to the appendix. From now on, when we refer to $r, R$ we take them to equal to the values specified in Lemma \ref{well-balancedness-norm}. We note that this choice may not be optimal\footnote{In fact using John's theorem one can easily show that the ratio between the radii of the smallest Euclidean ball containing $K$ divided by the radius of the largest Euclidean ball contained in $K$ is always $\leq \sqrt{n}$.}, nevertheless this does not affect the rates as we shall see.

\subsection{Estimator and Main Result}

Recall that for Gaussian noise the minimax rate as established in \cite{neykov2022minimax} for convex (and even star-shaped sets see \cite{prasadan2026information}) is given by $\eta^*$ defined in \eqref{minimax:rate}. We begin by establishing a simple result, relating the solution of the entropy equation to a minimization problem involving the entropy numbers.

\begin{lemma}\label{first:lemma}
    Let $n \geq 10$. We have that $\eta^{*2} \lceil\log_c (1 + d/\sigma)\rceil \gtrsim (\min_{k \in [\lfloor n/2\rfloor - 1]} e_k^2(K) + k \sigma^2) \wedge n \sigma^2 \wedge d^2$. On the other hand if $n < 10$ then $\eta^{*2} \asymp d^2 \wedge \sigma^2$.
\end{lemma}

\begin{proof}

First let $n \geq 10$.

Set $k^* = \lceil \log_2 M(\eta^*) \rceil + 1$. First, suppose $k^* \leq \lfloor n/2 \rfloor - 1$. Then we conclude
\begin{align*}
    \min_{k \in [\lfloor n/2\rfloor - 1]} e_k^2(K) + k\sigma^2 \leq e_{k^*}^2(K) +  k^*\sigma^2 \leq \eta^{*2} + (\log_2 M(\eta^*) + 2) \sigma^2,
\end{align*}
where we used $k^* - 1 = \lceil \log_2 M(\eta^*) \rceil \geq \log_2 M(\eta^*) \geq \log_2 N(\eta^*)$, and the definition of dyadic entropy numbers. Next we recall Lemma 3 of \cite{yang1999information}, which yields:
\begin{align}\label{yang:barron:bound}
    \log_2 M(\delta/c) - \log_2 M(\delta) \leq \log_2 M^{\operatorname{loc}}(\delta)
\end{align}
Setting $\delta_j = c^j \eta^*$ for $j = 1, \ldots, \lceil \log_c(d/\eta^*)\rceil$\footnote{If $\eta^* \geq d$, the result is immediate as the right hand side is at most $d^2$; hence we assume $\eta^ < d$.} and summing the corresponding inequalities, we obtain:
\begin{align}\label{logM:ineq}
    \log_2 M( \eta^*) &\leq \sum_{j = 1}^{\lceil \log_c(d/\eta^*)\rceil}  \log_2 M(\delta_j/c) - \log_2 M(\delta_j) \leq \sum_{j = 1}^{\lceil \log_c(d/\eta^*)\rceil}  \log_2 M^{\operatorname{loc}}(\delta_j) \nonumber \\
    & \lesssim \lceil \log_c(d/\eta^*)\rceil \eta^{*2}/\sigma^2,
\end{align}
where we used the non-increasing nature of local entropy, and the fact that $\log M(\delta_{\lceil \log_c(d/\eta^*)\rceil}) = 0$. Therefore
\begin{align*}
    (\min_{k \in [\lfloor n/2\rfloor - 1]} e_k^2(K) + k \sigma^2) \wedge n\sigma^2 \wedge d^2 & \lesssim \eta^{*2} + \lceil (\log_c(d/\eta^*)\rceil \eta^{*2} + 2 \sigma^2) \wedge d^2\\
    &\lesssim \eta^{*2} + \lceil \log_c(d/\eta^*)\rceil \eta^{*2} + (2 \sigma^2) \wedge d^2\\
    &\lesssim \lceil \log_c(d/\eta^*)\rceil \eta^{*2}, 
\end{align*}
where we used the fact that $\eta^* \gtrsim \sigma \wedge d$ \cite[see Lemma 1.4 of][e.g.]{prasadan2024some}.

Next suppose $k^* \geq \lfloor n/2\rfloor $. By \eqref{logM:ineq}, we immediately obtain $ \lceil \log_c(d/\eta^*)\rceil \eta^{*2} \geq (\lfloor n/2\rfloor-1)\sigma^2$. Thus for $n\geq 10$ we have 
\begin{align*}
    \lceil \log_c(d/\eta^*)\rceil \eta^{*2} & \gtrsim n\sigma^2 \geq (\min_{k \in [\lfloor n/2\rfloor - 1]} e_k^2(K) + k \sigma^2) \wedge n \sigma^2 \wedge d^2.
\end{align*}

If $n < 10$ the minimax rate is $\eta^{*2} \asymp d^2 \wedge \sigma^2$ (see for instance \cite[Lemma 1.4 of][e.g.]{prasadan2024some} for the lower bound, and the upper bound is clear upon noticing that the $\log M^{\operatorname{loc}}(\delta) \lesssim n$ thus $\eta^{*2} \lesssim n \sigma^2$ always and the upper bound of $d^2$ is obvious).
\end{proof}

We propose the following estimator, for a tuning parameter $\lambda > 0$:

\begin{align}\label{main:theory:estimator}
    \widehat \mu \in \argmin_{\nu \in K} \|Y - \nu\|_2^2 + \lambda \sigma^2 \|\nu\|_0 \log (e n/ \|\nu\|_0),
\end{align}
where with a slight abuse of notation we denote the cardinality of the non-zero entries of $\nu$ with $\|\nu\|_0$.\footnote{It may not be obvious why this estimator coincides with the one described in \eqref{estimator:def:introduction} in the introduction. However, this will  become apparent soon. } Here we use the convention that our penalty equals $0$ in case that $\nu = 0$ vector (i.e., $\|\nu\|_0 = 0$).

We first show that the estimator can be approximately computed in polynomial time under a single assumption on \(K\): we need oracle access to its Minkowski gauge. Furthermore, recall that $K$ is 
well-balanced, i.e., we can compute constants \(r,R>0\) such that
\begin{align*}
    rB_2 \subseteq K \subseteq RB_2.
\end{align*}

First let $T_s(Y)$ denote the set of indices corresponding to the largest in magnitude $|Y_i|$, breaking ties lexicographically. We define the at most $s$-sparse estimator $\widehat \mu_s \in \argmin_{\nu \in K, \nu_i = 0 \mbox{ \scriptsize for } i \in T^c_s(Y)} \|Y - \nu\|_2^2$. Next we have 
\begin{align}\label{theoretical:s:tilde}
    \widehat s = \argmin_{s \in \{0,\ldots, n\}} \|Y- \widehat \mu_s\|_2^2 + \lambda \sigma^2 s \log (en/s).
\end{align}

Assuming we can compute each of the $\widehat \mu_s$ for each $s \in \{0,\ldots, n\}$ the following lemma shows that $\widehat \mu = \widehat \mu_{\widehat s}$.

\begin{lemma}
    We have $\widehat \mu_{\widehat s} \in \argmin_{\nu \in K} \|Y - \nu\|_2^2 + \lambda \sigma^2 \|\nu\|_0 \log (e n/ \|\nu\|_0)$.
\end{lemma}
\begin{proof}
The statement follows from the following claim: for a fixed sparsity $s$, the solution to the problem $\widehat \mu_s \in \argmin_{\nu \in K, \|\nu\|_0\leq s} \|Y - \nu\|_2^2$. Observe that $\|Y - \nu\|_2^2 = \|Y\|_2^2 - 2 \langle Y, \nu \rangle + \|\nu\|_2^2$. Next, by the symmetry of $K$ we can permute and sign-flip all coordinates of $\nu$. Since $\|\cdot \|_2$ is also a symmetric norm this does not affect the term $\|\nu\|_2^2$, and the term $\|Y\|_2^2$ is just a constant. Hence  the claim follows by the rearrangement inequality.
\end{proof}

This lemma shows that to solve \eqref{main:theory:estimator}, there is no need to scan over all $2^n$ possible supports. Instead, we can solve the following $n$ problems: For $s \in [n]$\footnote{In fact it suffices to solve the problems on a dyadic grid $s \in \{1,2,4,\ldots, 2^{\lfloor \log_2 n \rfloor}\}$, but for simplicity we let $s$ range in $[n]$.}
\begin{align*}
      \widehat \mu_s \in \argmin \|Y - x\|^2_2, \mbox{ s.t. } x \in K, ~ \operatorname{supp}(x) \subseteq T_s(Y).
\end{align*}
Let  $K_s := \{x_{T_s(Y)} : x \in K, ~ \operatorname{supp}(x) \subseteq T_s(Y)\} \subset \RR^s$. Thus our optimization above can be written as:
\begin{align*}
     \widehat \mu_{s, T_s(Y)} \in \argmin \|Y_{T_s(Y)} - x\|^2_2, \mbox{ s.t. } x \in K_s.
\end{align*}

Since we know the set $K$ is balanced and its gauge is symmetric, it is clear that the convex set $K_s$ is also $(r, R)$ balanced. Further, it is clear that we can evaluate the gauge of $K_s$ efficiently. Since $\|.\|_2$ is a convex $1$-Lipschitz function, it follows by Theorem 2.5.9 of \cite{dadush2012integer} that there exists a polynomial-time algorithm, which can approximate the projection, i.e., we can compute a point $\tilde \mu_s \in K_s$ such that 
\begin{align*}
    \|Y_{T_s(Y)} - \tilde \mu_{s, T_s(Y)}\|_2 & \leq  \|Y_{T_s(Y)}  - \widehat \mu_{s, T_s(Y)}\|_2 + \epsilon,
\end{align*} 
for any $\epsilon > 0$ where the number of iterations required to achieve $\epsilon$-accuracy is proportional to $\log (R/\epsilon)$ (see also \cite{lee2018efficient} for a similar result). Set $\epsilon$ to an appropriate value ($\epsilon = (\sigma\wedge d) \wedge \frac{\sigma^2\wedge d^2}{\|Y\|_2 + R}$)\footnote{It is simple to see that setting $\epsilon$ to this value results in a number of iterations that depends polynomially on $n$ with high probability. In the case when $\|Y\|_2$ is too large one can output the $0$ point. This does not cause issues in the minimax rate since $\|Y\|_2 \lesssim R + \sqrt{n}\sigma$ with high probability.}.

In place of \eqref{theoretical:s:tilde}, let us define
\begin{align*}
    \tilde s = \argmin_{s \in \{0,\ldots, n\}} \|Y- \tilde \mu_s\|_2^2 + \lambda \sigma^2 s \log (en/s),
\end{align*}
where $\tilde \mu_s$ has the same non-zero entries as $\tilde \mu_{s, T_s(Y)}$ and $\operatorname{supp}(\tilde \mu_s) \subseteq T_s(Y)$. Next, define the computable estimator $\tilde \mu = \tilde \mu_{\tilde s}$. This is our estimator of $\mu$, \textbf{in the case when $\sigma \geq r/\sqrt{n}$}. \textbf{If $\sigma < r/\sqrt{n}$} we set $\tilde \mu = Y$. We have the following result regarding $\tilde \mu$.
\begin{theorem}[$\tilde \mu$ is Near-Optimal]\label{main:result} Let $\eta^*$ denote the critical radius defined in \eqref{minimax:rate}. Then if $\lambda \asymp c_0$ for some absolute constant $c_0 > 0$ we have:
\begin{align*}
    \sup_{\mu \in K} \EE_\mu \|\tilde \mu - \mu\|_2^2 \lesssim  \eta^{*2}\log(2n) \wedge d^2 \wedge n\sigma^2,
\end{align*}
where $\lesssim$ hides universal constants.
\end{theorem}
\begin{proof}[Proof of Theorem \ref{main:result}]
First, suppose $\sigma < r/\sqrt{n}$. As shown in \cite[Lemma 2.10]{neykov2025polynomial} in that case the rate $\eta^{*2} \asymp n \sigma^2$ and is therefore achieved exactly by $Y$ as $\EE_{\mu} \|Y - \mu\|_2^2 = \sum \operatorname{Var}(\xi_i) \leq n \sigma^2$.

Now we assume the more interesting case when $\sigma \geq r/\sqrt{n}$. For any $s \in [n]$ (for $s = 0$ $\tilde \mu_0 = \widehat \mu_0 = 0$):
\begin{align*}
    \|Y - \tilde \mu_s\|_2^2 & \leq  \|Y - \widehat \mu_s\|_2^2 + 2 \epsilon  \|Y - \widehat \mu_s\|_2 + \epsilon^2 \leq \|Y - \widehat \mu_s\|_2^2 + 2 \epsilon (\|Y\|_2 + \|\widehat \mu_s\|_2) + \epsilon^2 \\
    & \leq \|Y - \widehat \mu_s\|_2^2 + 3(\sigma^2\wedge d^2).
\end{align*}

For an integer $s \in \{0,\ldots, n\}$ define the shorthand $\operatorname{pen}(s) = \lambda \sigma^2 s \log (en/s)$ with the convention that $\operatorname{pen}(0) = 0$. Let \(x\in C_s\) be any \(s\)-sparse approximation to \(\mu\), where
\begin{align*}
    C_s:=K\cap\{z\in\mathbb R^n:\|z\|_0\le s\}.
\end{align*}
Let
\begin{align*}
    \widehat s:=\|\widehat\mu\|_0,
\end{align*}
where recall that $\widehat\mu$ was defined in \eqref{main:theory:estimator}. Since \(\widehat\mu\) minimizes the penalized criterion, we have
\begin{align*}
    \|Y-\widehat\mu\|_2^2+\operatorname{pen}(\widehat s)
    \le
    \|Y-x\|_2^2+\operatorname{pen}(s).
\end{align*}
Hence
\begin{align*}
    \|Y-\tilde\mu\|_2^2+\operatorname{pen}(\tilde s) 
    &\leq
    \|Y-\tilde\mu_{\widehat s}\|_2^2+\operatorname{pen}(\widehat s)
    \\&\le
    \|Y-\widehat \mu_{\widehat s}\|_2^2+\operatorname{pen}(\widehat s) + 3(\sigma^2\wedge d^2)
    \\&\le
    \|Y-x\|_2^2+\operatorname{pen}(s) + 3(\sigma^2\wedge d^2).
\end{align*} 

Using \(Y=\mu+\xi\), this becomes
\begin{align*}
    \|\tilde\mu-\mu-\xi\|_2^2+\operatorname{pen}(\tilde s)
    \le
    \|x-\mu-\xi\|_2^2+\operatorname{pen}(s) + 3(\sigma^2\wedge d^2).
\end{align*}
Expanding both squared norms gives
\begin{align*}
    \|\tilde\mu-\mu\|_2^2
    -2\langle \xi,\tilde\mu-\mu\rangle
    +\|\xi\|_2^2
    +\operatorname{pen}(\tilde s)
    \le
    \|x-\mu\|_2^2
    -2\langle \xi,x-\mu\rangle
    +\|\xi\|_2^2
    +\operatorname{pen}(s) + 3(\sigma^2\wedge d^2).
\end{align*}
Canceling \(\|\xi\|_2^2\) from both sides yields
\begin{align*}
    \|\tilde\mu-\mu\|_2^2
    +\operatorname{pen}(\tilde s)
    \le
    \|x-\mu\|_2^2
    +\operatorname{pen}(s)
    +2\langle \xi,\tilde\mu-x\rangle + 3(\sigma^2\wedge d^2).
\end{align*}
Thus the basic inequality is
\begin{align*}
    \boxed{
    \|\tilde\mu-\mu\|_2^2
    \le
    \|x-\mu\|_2^2
    +\operatorname{pen}(s)
    -\operatorname{pen}(\tilde s)
    +2\langle \xi,\tilde\mu-x\rangle + 3(\sigma^2\wedge d^2).
    }
\end{align*}

It remains to control the stochastic term. Let
\begin{align*}
    S:=\operatorname{supp}(x),
    \qquad
    \tilde S:=\operatorname{supp}(\tilde\mu).
\end{align*}
Then
\begin{align*}
    \operatorname{supp}(\tilde\mu-x)\subseteq S\cup \tilde S,
\end{align*}
and hence
\begin{align*}
    |S\cup \tilde S|\le s+\tilde s.
\end{align*}
Therefore
\begin{align*}
    \langle \xi,\tilde\mu-x\rangle
    =
    \langle P_{S\cup\tilde S}\xi,\tilde\mu-x\rangle
    \le
    \|P_{S\cup\tilde S}\xi\|_2\,
    \|\tilde\mu-x\|_2,
\end{align*}
with $P_{S\cup \tilde S}$ being the projection on the union of the two support sets $S$ and $\tilde S$. Using \(2ab\le \kappa a^2+\kappa^{-1}b^2\), for any \(\kappa>0\),
\begin{align*}
    2\langle \xi,\tilde\mu-x\rangle
    \le
    \kappa\|\tilde\mu-x\|_2^2
    +
    \kappa^{-1}\|P_{S\cup\tilde S}\xi\|_2^2.
\end{align*}
Moreover,
\begin{align*}
    \|\tilde\mu-x\|_2^2
    \le
    2\|\tilde\mu-\mu\|_2^2
    +
    2\|x-\mu\|_2^2.
\end{align*}
Thus
\begin{align*}
    2\langle \xi,\tilde\mu-x\rangle
    \le
    2\kappa\|\tilde\mu-\mu\|_2^2
    +
    2\kappa\|x-\mu\|_2^2
    +
    \kappa^{-1}\|P_{S\cup\tilde S}\xi\|_2^2.
\end{align*}
Substituting this into the basic inequality gives
\begin{align*}
    \|\tilde\mu-\mu\|_2^2
    \le
    \|x-\mu\|_2^2
    +\operatorname{pen}(s)
    -\operatorname{pen}(\tilde s)
    +
    2\kappa\|\tilde\mu-\mu\|_2^2
    +
    2\kappa\|x-\mu\|_2^2
    +
    \kappa^{-1}\|P_{S\cup\tilde S}\xi\|_2^2 + 3(\sigma^2\wedge d^2).
\end{align*}
Choosing, for example, \(\kappa=1/4\), and absorbing the term
\begin{align*}
    2\kappa\|\tilde\mu-\mu\|_2^2
\end{align*}
into the left-hand side, we obtain
\begin{align}\label{cool:inequality}
    \|\tilde\mu-\mu\|_2^2
    \lesssim
    \|x-\mu\|_2^2
    +
    \operatorname{pen}(s)
    -
    \operatorname{pen}(\tilde s)
    +
    \|P_{S\cup\tilde S}\xi\|_2^2 +  (\sigma^2\wedge d^2).
\end{align}
This is the desired deterministic reduction: the estimation error is controlled
by the approximation error, the penalty difference, and the noise
energy on the union of the true comparison support and the selected support.

Recall that we are assuming that the coordinates of \(\xi\) are
independent, mean-zero, sub-Gaussian random variables with sub-Gaussian parameter $\leq \sigma$.
We now apply the Hanson--Wright inequality: there exist universal constants
\(c,C>0\) such that, for every fixed matrix \(A\),
\begin{align*}
    \mathbb P\left(
        \left|
            \xi^\top A\xi-\mathbb E \xi^\top A\xi
        \right|
        >
        C\sigma^2\left(
            \|A\|_{\mathrm F}\sqrt t
            +
            \|A\|_{\mathrm{op}}t
        \right)
    \right)
    \le
    2e^{-ct}.
\end{align*}
Apply this with \(A=P_T\), where \(T\subseteq[n]\) and \(|T|=m\). Then
\begin{align*}
    \|P_T\|_{\mathrm F}=\sqrt m,
    \qquad
    \|P_T\|_{\mathrm{op}}=1,
\end{align*}
and
\begin{align*}
    \mathbb E\|P_T\xi\|_2^2
    =
    \sum_{j\in T}\mathbb E\xi_j^2
    \le
    \sigma^2 m .
\end{align*}
Therefore, for every fixed \(T\) with \(|T|=m\),
\begin{align*}
    \mathbb P\left(
        \|P_T\xi\|_2^2
        >
        \sigma^2 m
        +
        C\sigma^2\{\sqrt{mt}+t\}
    \right)
    \le
    2e^{-ct}.
\end{align*}
Equivalently, after changing the universal constant \(C\),
\begin{align*}
    \mathbb P\left(
        \|P_T\xi\|_2^2
        >
        C\sigma^2\{m+\sqrt{mt}+t\}
    \right)
    \le
    2e^{-ct}.
\end{align*}
Since \(\sqrt{mt}\le (m+t)/2\), this implies
\begin{align*}
    \mathbb P\left(
        \|P_T\xi\|_2^2
        >
        C\sigma^2(m+t)
    \right)
    \le
    2e^{-ct}.
\end{align*}

We now make this bound uniform over all supports. Fix \(u>0\). For sets
\(T\) with \(|T|=m\), choose
\begin{align*}
    t_{m,u}
    :=
    c^{-1}\left\{
        u+\log\binom{n}{m}+m+\log 2
    \right\}.
\end{align*}
Then
\begin{align*}
\begin{aligned}
    &\mathbb P\left(
        \exists\, T\subseteq[n],\ |T|=m\ge 1:
        \|P_T\xi\|_2^2
        >
        C\sigma^2(m+t_{m,u})
    \right) \\
    &\qquad\le
    \sum_{m=1}^n
    2\binom{n}{m}e^{-ct_{m,u}} \\
    &\qquad=
    \sum_{m=1}^n
    e^{-u}e^{-m}
    \le
    e^{-u}.
\end{aligned}
\end{align*}
Hence, with probability at least \(1-e^{-u}\), uniformly over all
nonempty \(T\subseteq[n]\), writing \(m=|T|\),
\begin{align*}
    \|P_T\xi\|_2^2
    \le
    C\sigma^2
    \left\{
        m
        +
        \log\binom{n}{m}
        +
        u
    \right\},
\end{align*}
where we have absorbed the $\log 2$ term in the $m$ term by adjusting the constant and assuming $m \geq 1$. Using
\begin{align*}
    \log \binom{n}{m}
    \le
    m\log(en/m)
    =:
    \phi(m),
    \qquad
    m\le \phi(m),
\end{align*}
we obtain the uniform bound
\begin{align*}
    \|P_T\xi\|_2^2
    \le
    C\sigma^2\{\phi(|T|)+u\}
\end{align*}
simultaneously for all \(T\subseteq[n]\) with $|T| \geq 1$.

Now let \(S=\operatorname{supp}(x)\) and
\(\tilde S=\operatorname{supp}(\tilde\mu)\), with
\begin{align*}
    |S|\le s,
    \qquad
    |\tilde S|\le \tilde s .
\end{align*}
Applying the preceding uniform bound to \(T=S\cup\tilde S\), we get
\begin{align*}
    \|P_{S\cup\tilde S}\xi\|_2^2
    \le
    C\sigma^2
    \left\{
        \phi(|S\cup\tilde S|)
        +
        u
    \right\}.
\end{align*}
Since \(|S\cup\tilde S|\le n \wedge(s+\tilde s)\), we have $\phi(|S\cup\tilde S|) \leq \phi(s) + \phi(\tilde s)$. Indeed if $s + \tilde s \leq n$ this follows by monotonicity and subadditivity of $\phi$; else if $s + \tilde s \geq n$ we have
\begin{align*}
   \phi(|S\cup\tilde S|)
    \le
    \phi(n) = n < s + \tilde s \leq \phi(s) + \phi(\tilde s),
\end{align*}
we obtain
\begin{align*}
    \|P_{S\cup\tilde S}\xi\|_2^2
    \le
    C\sigma^2
    \left\{
        s\log(en/s)
        +
        \tilde s\log(en/\tilde s)
        +
        u
    \right\}.
\end{align*}
Consequently, if the basic inequality contains the stochastic term with a
fixed numerical coefficient \(c_0>0\), then on the preceding event,
\begin{align*}
    c_0\|P_{S\cup\tilde S}\xi\|_2^2
    \le
    Cc_0\sigma^2
    \left\{
        s\log(en/s)
        +
        \tilde s\log(en/\tilde s)
        +
        u
    \right\}.
\end{align*}
Therefore, for the penalty
\begin{align*}
    \operatorname{pen}(m)
    =
    \lambda\sigma^2 m\log(en/m),
\end{align*}
choosing \(\lambda > Cc_0\) ensures that the term depending on
\(\tilde s\) is absorbed by \(-\operatorname{pen}(\tilde s)\).
Thus, by \eqref{cool:inequality} and the above logic with probability at least \(1-e^{-u}\),
\begin{align*}
    \|\tilde\mu-\mu\|_2^2
    \lesssim
    \|x-\mu\|_2^2
    +
    \sigma^2 s\log(en/s)
    +
    \sigma^2 u + (\sigma^2\wedge d^2).
\end{align*}

Equivalently,
\begin{align*}
    \mathbb P_\mu\left(
        \|\tilde\mu-\mu\|_2^2
        \gtrsim
        \|x-\mu\|_2^2
        +
        \sigma^2 s\log(en/s)
        + \sigma^2\wedge d^2 + 
        \sigma^2 u
    \right)
    \le e^{-u}.
\end{align*}
Integrating this tail bound in \(u\) gives
\begin{align*}
    \mathbb E_\mu \|\tilde\mu-\mu\|_2^2
    \lesssim
    \|x-\mu\|_2^2
    +
    \sigma^2 s\log(en/s),
\end{align*}
for $s \geq 1$. In other words,
since \(x\in C_s\) was arbitrary we have,
\begin{align}\label{important:intermediate:inequality}
    \mathbb E_\mu \|\tilde\mu-\mu\|_2^2 \lesssim \inf_{1\le s\le n}
    \left\{
        \inf_{x\in C_s}\|\mu-x\|_2^2 + \sigma^2 s\log(en/s)
    \right\}.
\end{align}
where recall the definition of $C_s$:
\begin{align*}
    C_s:= K \cap \{x\in\mathbb R^n:\|x\|_0\le s\}.
\end{align*}

Suppose now that $n < 10$. \eqref{important:intermediate:inequality} implies that $\EE_\mu \|\tilde\mu-\mu\|_2^2 \lesssim (d^2 \wedge n\sigma^2) \asymp d^2 \wedge \sigma^2$ which is the minimax rate as we explained in Lemma \ref{first:lemma}. Hence we now assume $n \geq 10$.

We now recall a seminal result due to Edmunds and Netrusov. The result states that for $k < n/2$, the $k$-th dyadic entropy number $e_k(K) \asymp u_s(K)$ where $s = \left\lceil \frac{k}{\log(n/k + 1)} \right \rceil$, and $u_s(K)$ is the following expression:
\begin{align*}
    u_s(K) = \sup_{a \in K} \left( \sum_{j=1}^n \min\{a_j^*,a_s^*\}^2\right)^{1/2},
\end{align*}
where $a^*$ denotes a decreasing rearrangement of the absolute values of the entries of the vector $a$: $a_1^* \geq a_2^* \geq \ldots \geq a_n^* \geq 0$. 

We immediately observe that the quantity $u_s(K)$ is related to soft-thresholding. Specifically, set $\tau = a^*_s$ where $a^*$ denotes a decreasing rearrangement of $|a|$ (where $|\cdot|$ is applied entrywise) of the vector $a \in \RR^n$, and note that $\min (|a_i|, \tau) = |a_i| - (|a_i| - \tau)_+$. Thus if $S_\tau(a)$ is a vector given coordinate-wise as $S_\tau(a)_i = \sign(a_i)(|a_i| - \tau)_+$, then
\begin{align*}
    u_s(K) = \sup_{\|a\|_K \leq 1}\|a - S_\tau(a)\|_2
\end{align*}

Now let us consider the related hard-thresholding quantity

\begin{align*}
    h_s(a) = (\sum_{j > s} (a_j^*)^2)^{1/2},
\end{align*}
while
\begin{align*}
    u_s(a) = (\sum_{j \in [n]} (a_j^*\wedge a_s^*)^2)^{1/2},
\end{align*}

Then $u^2_s(a) = h^2_s(a) + s (a_s^*)^{2}$ for any $s \geq 1$. Thus, clearly $u_s(a) \geq h_s(a)$ for $s \geq 1$. This immediately implies that $u_s(K) = \sup_{a \in K} u_s(a) \geq \sup_{a \in K} h_s(a) =: h_s(K)$, where we denoted with $h_s(K)$ the worst case best $s$-sparse approximation of any vector in $K$.

Clearly, $\sup_{\mu \in K}\inf_{x\in C_s}\|\mu-x\|_2^2 \leq\sup_{\mu \in K} h^2_s(\mu)  \leq u^2_s(K)$ for $s \geq 1$. Observe that on the other hand since $\tilde \mu$ is a proper estimator (when $\sigma > r/\sqrt{n}$), it also follows that $ \mathbb E_\mu \|\tilde\mu-\mu\|_2^2 \lesssim d^2$, and thus the above inequality extends to the case when $s = 0$.  Furthermore, as we mentioned above when $s = n$ we have the RHS of \eqref{important:intermediate:inequality} is $n\sigma^2$. Therefore, since \(x\in C_s\) and \(s\) were arbitrary,
\begin{align*}
    \mathbb E_\mu \|\tilde\mu-\mu\|_2^2
    \lesssim
    \inf_{1\le s\le {n}}
    \left\{
        (u_s^2(K)
        +
        \sigma^2 s\log(en/s))\wedge d^2 \wedge (n\sigma^2)
    \right\} 
\end{align*}

We now state a simple calculation whose proof is deferred to the appendix.

\begin{lemma}\label{simple:calculation}
Let \(1\le k\le n\), and define
\begin{align*}
    s:=\left\lceil \frac{k}{\log(1+n/k)} \right\rceil .
\end{align*}
Then
\begin{align*}
    s\log\left(\frac{en}{s}\right)
    \asymp
    k+\log(en),
\end{align*}
where the implicit constants are universal. In particular, if $k \gtrsim \log(en)$,
then
\begin{align*}
    s\log\left(\frac{en}{s}\right)
    \asymp k .
\end{align*}
\end{lemma}

The lemma above in conjunction with Edmunds and Netrusov's theorem implies that 
\begin{align*}
    \EE \|\tilde \mu - \mu\|_2^2 \lesssim \inf_{\log (en)\le k < {n/2}}
    \left\{
        (e_k^2(K) + \sigma^2 k)\wedge d^2 \wedge (n\sigma^2)
    \right\} 
\end{align*}

Let $ k^* = \argmin_{1 \le k < n/2}(e_k^2(K) + \sigma^2 k)$ and $k^{**} = \argmin_{\log (en)\le k < n/2}(e_k^2(K) + \sigma^2 k)$.

Let us compare the two expressions $e_{k^*}^2(K) + \sigma^2 k^*$ vs $e_{k^{**}}^2(K) + \sigma^2 k^{**}$. If $k^* \geq \log (en)$ then $k^* = k^{**}$ and there is no difference between the two.

Suppose now that $k^* < \lceil \log (en) \rceil$. Then $e_{k^*}^2(K) \geq e_{\lceil \log (en) \rceil}^2(K)$. Since $k^* \geq 1$ we have
\begin{align*}
    e_{k^*}^2(K) + k^* \sigma^2 \geq e_{k^*}^2(K) + \sigma^2.
\end{align*}
On the other hand
\begin{align*}
    e_{k^{**}}^2(K) + \sigma^2 k^{**} \leq e_{\lceil \log (en) \rceil}^2(K) + \lceil \log (en) \rceil\sigma^2 \leq e_{k^*}^2(K) + \lceil \log (en) \rceil\sigma^2,
\end{align*}
implying that the difference between the two expressions is at most $\log(en)\sigma^2$. Hence we conclude
\begin{align*}
    \EE \|\tilde \mu - \mu\|_2^2 & \lesssim \inf_{\log (en)\le k <  n/2}
    \left\{
        (e_k^2(K) + \sigma^2 k)\wedge d^2 \wedge (n\sigma^2)
    \right\} \\
    & \leq \inf_{1\le k < n/2}
    \left\{
        (e_k^2(K) + \sigma^2 k)\wedge d^2 \wedge (n\sigma^2) 
    \right\} + (\log(en)\sigma^2)\wedge d^2 \wedge n \sigma^2.
\end{align*}
Now by Lemma \ref{first:lemma} (recall we are assuming $n \geq 10$) we have
\begin{align*}
    \EE \|\tilde \mu - \mu\|_2^2 & \lesssim \eta^{*2}\log(1 + d/\sigma) \wedge d^2 \wedge n\sigma^2 + (\log(en)\sigma^2 )\wedge d^2 \wedge n\sigma^2
\end{align*}
Recall now that we are assuming $\sigma \geq r/\sqrt{n}$. In this case we can further bound RHS above as:
\begin{align*}
    \EE \|\tilde \mu - \mu\|_2^2 & \lesssim \eta^{*2}\log(1 + 2\sqrt{n}R/r) \wedge d^2 \wedge n\sigma^2 + (\log(en)\sigma^2 )\wedge d^2 \wedge n\sigma^2,
\end{align*}
To this end recall the definitions of $r$ and $R$ in Lemma \ref{well-balancedness-norm}. We conclude
\begin{align*}
    \EE \|\tilde \mu - \mu\|_2^2 & \lesssim \eta^{*2}\log( 1 + 2n^{3/2}) \wedge d^2 \wedge n\sigma^2 + (\log(en)\sigma^2 )\wedge d^2 \wedge n \sigma^2.
\end{align*}
As shown in \cite{prasadan2024some} $\eta^* \gtrsim \sigma \wedge d$ and therefore
\begin{align*}
    \EE \|\tilde \mu - \mu\|_2^2 & \lesssim \eta^{*2}\log(2n) \wedge d^2 \wedge n\sigma^2.
\end{align*}
which is what we aimed to show.
\end{proof}

\section{Linear Regression}\label{regression:section}
In this section we extend our main result to the linear regression setting. Let
\begin{align*}
    Y=X\beta+\xi,\qquad
    \widehat \Sigma:=\frac{X^\top X}{N},
\end{align*}
and suppose that the entries \(\xi_1,\dots,\xi_N\)
are i.i.d., mean-zero, sub-Gaussian random variables with sub-Gaussian parameter at most $\sigma$ where $\sigma \asymp 1$ is a \textbf{fixed, positive} constant. 

Let $\beta \in K \subseteq \RR^p$ where $K$ has a symmetric Minkowski functional as detailed in Definition \ref{symmetric:norm:def}. We assume a ``moderate dimensional setting'': $N \gtrsim p$ for some absolute constant. We suppose that the covariates (i.e. the rows of the matrix $X$) $X_i$ are independent from the noise $\xi_i$. We also assume that the diameter of $K$ is bounded above by an absolute constant, i.e., $d := \diam(K) = O(1)$. We will remark later why this assumption is not very restrictive in the moderate dimensional regime. 

Suppose the covariates $X_i$ are centered\footnote{If the predictors $X_i$ are not centered one can consider $(Y_{2i} - Y_{2i-1}, X_{2i} - X_{2i-1})_{i \in [\lfloor N/2\rfloor]}$ to center the predictors while leaving the $\beta$ unchanged. We note that this operation prevents the model from having an intercept.} sub-Gaussian variables with a well conditioned covariance matrix $\EE X_i X_i\T = \Sigma$, i.e., $c' \leq \lambda_{\min}(\Sigma) \leq \lambda_{\max}(\Sigma) \leq C'$ for some $c',C' > 0$, and sub-Gaussian parameter of constant order, i.e. $\sup_{v \in S^{p-1}} \EE \exp(\lambda v\T X_i) \leq \exp(\lambda^2 \zeta^2/2)$ for $\zeta = O(1)$. Then the minimax rate $\eta^{*2}$, in the case when $\xi_i \sim N(0,1)$, is given by the entropy equation 
\begin{align*}
    \eta^* = \sup_\eta\{\eta \geq 0: N\eta^2 \leq \log M_K^{\operatorname{loc}}(\eta)\},
\end{align*} 
as shown in \cite{prasadan2025characterizingminimaxratenonparametric}. 

Define the effective noise vector
\begin{align*}
    \gamma
    :=
    \widehat \Sigma^{-1}\frac{X^\top \xi}{N},
\end{align*}
assuming that the matrix $\widehat \Sigma$ is invertible. In case $\hat \Sigma$ is singular, our estimator outputs an arbitrary point in the set $K$. Observe that the transformation
\begin{align*}
    \widehat \Sigma^{-1} X^\top  Y/N = \beta + \gamma,
\end{align*}
reduces the problem to a Gaussian sequence model setting. However, $\gamma$ is not comprised of independent sub-Gaussian components as required by the theory in Section \ref{main:section}. Since we used the ranodmness of $\gamma$ only when using the Hanson-Wright inequality this is the main place where we need to modify the proof. In addition, we need to take into account the case when the matrix $\widehat \Sigma$ is singular. 

Denote with $\tilde \beta$ the estimate, $\tilde \beta= 0$ if $\widehat \Sigma$ is not invertible, and $\tilde \beta = \tilde \mu(\widehat \Sigma^{-1} X^\top  Y/N)$ with penalty chosen as $\lambda\frac{\sigma^2}{N}m\log(ep/m)$, with $\lambda$ sufficiently large. We have
\begin{theorem}\label{linear:regression:theorem} Let $N \gtrsim p$. Under the assumptions above, and if $\lambda \asymp c_0$ for some absolute constant $c_0 > 0$ we have:
    \begin{align*}
        \EE \|\tilde \beta - \beta\|_2^2\lesssim \eta^{*2} \log(2p) \wedge d^2 \wedge \frac{p}{N},
    \end{align*}
    where $\eta^*$ is the critical radius defined above.
\end{theorem}

\begin{proof}[Proof of Theorem \ref{linear:regression:theorem}] We will simply show the modification required to show the Hanson-Wright argument.
    For a coordinate set \(T\subseteq[p]\), let \(P_T\) denote the coordinate
projection onto \(T\), and write \(m=|T|\). Then
\begin{align*}
    \|P_T\gamma\|_2^2
    =
    \xi^\top B_T\xi,
\end{align*}
where
\begin{align*}
    B_T
    :=
    \frac{1}{N^2}
    X\widehat \Sigma^{-1}P_T\widehat \Sigma^{-1}X^\top .
\end{align*}
Under our assumptions $\widehat \Sigma$ is invertible with high probability (see below for a precise quantification of this statement).

We apply the Hanson--Wright inequality conditionally on \(X\). That is,
there exist universal constants \(c,C>0\) such that, for every \(t>0\),
\begin{align*}
    \mathbb P\left(
        \xi^\top B_T\xi
        >
        \mathbb E[\xi^\top B_T\xi\mid X]
        +
        C\sigma^2\left\{
            \|B_T\|_{\mathrm F}\sqrt t
            +
            \|B_T\|_{\mathrm{op}}t
        \right\}
        \,\middle|\, X
    \right)
    \le
    2e^{-ct}.
\end{align*}

We now bound the trace, Frobenius norm, and operator norm of \(B_T\).
First,
\begin{align*}
\begin{aligned}
    \operatorname{tr}(B_T)
    &=
    \frac{1}{N^2}
    \operatorname{tr}\left(
        X\widehat \Sigma^{-1}P_T\widehat \Sigma^{-1}X^\top
    \right)  \\
    &=
    \frac{1}{N^2}
    \operatorname{tr}\left(
        \widehat \Sigma^{-1}P_T\widehat \Sigma^{-1}X^\top X
    \right)  \\
    &=
    \frac{1}{N}
    \operatorname{tr}\left(
        \widehat \Sigma^{-1}P_T
    \right) \\
    &\le
    \frac{m}{N}\|\widehat \Sigma^{-1}\|_{\mathrm{op}} .
\end{aligned}
\end{align*}
Hence
\begin{align*}
    \mathbb E[\xi^\top B_T\xi\mid X]
    \le
    \sigma^2\operatorname{tr}(B_T)
    \le
    \sigma^2\frac{m}{N}\|\widehat \Sigma^{-1}\|_{\mathrm{op}} .
\end{align*}

Next, set
\begin{align*}
    Q:=\frac{1}{\sqrt N}X\widehat \Sigma^{-1/2}.
\end{align*}
Then \(Q^\top Q=I_p\), and
\begin{align*}
    B_T
    =
    \frac{1}{N}
    Q\widehat \Sigma^{-1/2}P_T\widehat \Sigma^{-1/2}Q^\top .
\end{align*}
Therefore,
\begin{align*}
    \|B_T\|_{\mathrm{op}}
    \le
    \frac{1}{N}
    \|\widehat \Sigma^{-1/2}P_T\widehat \Sigma^{-1/2}\|_{\mathrm{op}}
    \le
    \frac{1}{N}\|\widehat \Sigma^{-1}\|_{\mathrm{op}},
\end{align*}
and
\begin{align*}
\begin{aligned}
    \|B_T\|_{\mathrm F}
    &\le
    \frac{1}{N}
    \|\widehat \Sigma^{-1/2}P_T\widehat \Sigma^{-1/2}\|_{\mathrm F}  \\
    &\le
    \frac{1}{N}
    \|\widehat \Sigma^{-1}\|_{\mathrm{op}}\|P_T\|_{\mathrm F} \\
    &=
    \frac{\sqrt m}{N}\|\widehat \Sigma^{-1}\|_{\mathrm{op}} .
\end{aligned}
\end{align*}

Combining these estimates with Hanson--Wright gives, for every fixed
\(T\subseteq[p]\) with \(|T|=m\),
\begin{align*}
    \mathbb P\left(
        \|P_T\gamma\|_2^2
        >
        C\sigma^2
        \frac{\|\widehat \Sigma^{-1}\|_{\mathrm{op}}}{N}
        \{m+\sqrt{mt}+t\}
        \,\middle|\,X
    \right)
    \le
    2e^{-ct}.
\end{align*}
Since \(\sqrt{mt}\le (m+t)/2\), after changing the constant \(C\),
\begin{align*}
    \mathbb P\left(
        \|P_T\gamma\|_2^2
        >
        C\sigma^2
        \frac{\|\widehat \Sigma^{-1}\|_{\mathrm{op}}}{N}
        (m+t)
        \,\middle|\,X
    \right)
    \le
    2e^{-ct}.
\end{align*}

Now define
\begin{align*}
    \phi(m):=m\log(ep/m),\qquad 1\le m\le p,
\end{align*}
and set \(\operatorname{pen}(0)=0\). To make the bound uniform over all supports, fix
\(u>0\), and for sets \(T\) with \(|T|=m\), choose
\begin{align*}
    t_{m,u}
    :=
    c^{-1}
    \left\{
        u+\log\binom{p}{m}+m+\log 2
    \right\}.
\end{align*}
By a union bound,
\begin{align*}
\begin{aligned}
    &\mathbb P\left(
        \exists\, T\subseteq[p],\ |T|=m\ge 1:
        \|P_T\gamma\|_2^2
        >
        C\sigma^2
        \frac{\|\widehat \Sigma^{-1}\|_{\mathrm{op}}}{N}
        (m+t_{m,u})
        \,\middle|\, X
    \right) \\
    &\qquad\le
    \sum_{m=1}^p
    2\binom{p}{m}e^{-ct_{m,u}}
    \le
    e^{-u}.
\end{aligned}
\end{align*}
Thus, with conditional probability at least \(1-e^{-u}\), uniformly over
all \(T\subseteq[p]\),
\begin{align*}
    \|P_T\gamma\|_2^2
    \le
    C\sigma^2
    \frac{\|\widehat \Sigma^{-1}\|_{\mathrm{op}}}{N}
    \left\{
        |T|\log(ep/|T|)
        +
        u
    \right\}.
\end{align*}

Applying this bound to \(T=S\cup\widehat S\), where
\begin{align*}
    |S|\le s,
    \qquad
    |\widehat S|\le \widehat s,
\end{align*}
we obtain
\begin{align*}
    \|P_{S\cup\widehat S}\gamma\|_2^2
    \le
    C\sigma^2
    \frac{\|\widehat \Sigma^{-1}\|_{\mathrm{op}}}{N}
    \left\{
        \phi(|S\cup\widehat S|)
        +
        u
    \right\}.
\end{align*}
Since
\begin{align*}
    |S\cup\widehat S|\le p \wedge (s+\widehat s),
\end{align*}
we have $\phi(|S\cup\widehat S|) \leq \phi(s)+\phi(\widehat s)$.
Indeed if $s + \widehat s \leq p$, this follows from monotonicity and subadditivity of \(\phi\), while if $s + \widehat s \geq p$:
\begin{align*}
    \phi(|S\cup\widehat S|)
    \leq \phi(p) \leq 
    s + \widehat s \leq 
    \phi(s)+\phi(\widehat s).
\end{align*}
Hence we get
\begin{align*}
    \|P_{S\cup\widehat S}\gamma\|_2^2
    \le
    C\sigma^2
    \frac{\|\widehat \Sigma^{-1}\|_{\mathrm{op}}}{N}
    \left\{
        s\log(ep/s)
        +
        \widehat s\log(ep/\widehat s)
        +
        u
    \right\}.
\end{align*}

Therefore, the corresponding sparsity penalty should be chosen as
\begin{align*}
    \operatorname{pen}(m)
    =
    \lambda\sigma^2
    \frac{\|\widehat \Sigma^{-1}\|_{\mathrm{op}}}{N}
    m\log(ep/m),
\end{align*}
with \(\lambda>0\) sufficiently large. 

It remains to control
\(\|\widehat \Sigma^{-1}\|_{\mathrm{op}}\). Recall that
\[
    \widehat \Sigma
    :=
    \frac{1}{N}X^\top X
    =
    \frac{1}{N}\sum_{i=1}^N X_iX_i^\top .
\]
Let
\[
    Z_i:=\Sigma^{-1/2}X_i .
\]
Then \(\EE Z_iZ_i^\top=I_p\). Moreover, since
\(\lambda_{\min}(\Sigma)\ge c'\) and \(X_i\) has constant-order
sub-Gaussian norm, the vectors \(Z_i\) are isotropic sub-Gaussian with
sub-Gaussian norm bounded by a constant depending only on \(c'\) and the
original sub-Gaussian parameter.

By the standard singular-value bound for matrices with independent
isotropic sub-Gaussian rows \cite{vershynin2010introduction}, there exist constants \(c,C>0\) such that,
for every \(t\ge 0\), with probability at least \(1-2e^{-ct^2}\),
\[
    \sqrt N - C\sqrt p - t
    \le
    s_{\min}(Z)
    \le
    s_{\max}(Z)
    \le
    \sqrt N + C\sqrt p + t .
\]
Equivalently,
\[
    \lambda_{\min}\left(\frac{1}{N}Z^\top Z\right)
    \ge
    \left(
        1 - C\sqrt{\frac{p}{N}} - \frac{t}{\sqrt N}
    \right)^2
\]
and
\[
    \lambda_{\max}\left(\frac{1}{N}Z^\top Z\right)
    \le
    \left(
        1 + C\sqrt{\frac{p}{N}} + \frac{t}{\sqrt N}
    \right)^2 .
\]
Since
\[
    \widehat \Sigma
    =
    \Sigma^{1/2}
    \left(\frac{1}{N}Z^\top Z\right)
    \Sigma^{1/2},
\]
we get
\[
    \lambda_{\min}(\widehat \Sigma)
    \ge
    \lambda_{\min}(\Sigma)
    \left(
        1 - C\sqrt{\frac{p}{N}} - \frac{t}{\sqrt N}
    \right)^2 .
\]
Therefore, using \(\lambda_{\min}(\Sigma)\ge c'\),
\[
    \lambda_{\min}(\widehat \Sigma)
    \ge
    c'
    \left(
        1 - C\sqrt{\frac{p}{N}} - \frac{t}{\sqrt N}
    \right)^2 .
\]
Consequently,
\[
    \|\widehat \Sigma^{-1}\|_{\mathrm{op}}
    =
    \frac{1}{\lambda_{\min}(\widehat \Sigma)}
    \le
    \frac{1}{c'}
    \left(
        1 - C\sqrt{\frac{p}{N}} - \frac{t}{\sqrt N}
    \right)^{-2},
\]
provided the quantity in parentheses is positive.

In particular, if
\[
    C\sqrt{\frac{p}{N}}+\frac{t}{\sqrt N}
    \le \frac12,
\]
then
\[
    \|\widehat \Sigma^{-1}\|_{\mathrm{op}}
    \le
    \frac{4}{c'}.
\]
Thus, whenever \(N\gtrsim p+t^2\), the empirical covariance matrix is
well-conditioned with high probability, and
\[
    \|\widehat \Sigma^{-1}\|_{\mathrm{op}}
    \lesssim \frac{1}{c'} .
\]
On this event, the preceding conditional Hanson--Wright bound yields,
uniformly over all \(T\subseteq[p]\),
\[
    \|P_T\gamma\|_2^2
    \le
    C\sigma^2
    \frac{1}{Nc'}
    \left\{
        |T|\log(ep/|T|)
        +
        u
    \right\}.
\]
Therefore, applying this to \(T=S\cup\widehat S\), we obtain
\[
    \|P_{S\cup\widehat S}\gamma\|_2^2
    \le
    C\sigma^2
    \frac{1}{Nc'}
    \left\{
        s\log(ep/s)
        +
        \widehat s\log(ep/\widehat s)
        +
        u
    \right\}.
\]
Thus, up to constants depending only on \(c'\), the sparsity penalty may
be chosen as
\[
    \operatorname{pen}(m)
    =
    \lambda\frac{\sigma^2}{N}m\log(ep/m),
\]
for \(\lambda>0\) sufficiently large.

What is more, the bound on the sample covariance matrix of the covariates shows that it is invertible with probability at least $1 - \exp(-CN)$ whenever $N \gtrsim p$. It follows that we have error $d^2 = O(1)$ with probability at most $\exp(-CN)$. On the other hand the smallest possible value of the minimax rate is $d^2 \wedge 1/N$. Since $d^2 \cdot\exp(-CN) \ll d^2 \wedge 1/N$ this term can be ``folded'' in the minimax rate. This completes the proof.

\end{proof}

\begin{remark}
    Let us briefly explain why the bounded-diameter assumption is not very restrictive
in the moderate-dimensional regime. Suppose \(N \gtrsim p\), and let
\[
    \widehat \beta_{\rm OLS}
    =
    \widehat\Sigma^{-1}X^\top Y/N
\]
on the event that \(\widehat\Sigma\) is invertible. Since
\[
    \widehat \beta_{\rm OLS}-\beta
    =
    \widehat\Sigma^{-1}X^\top \xi/N,
\]
the same Hanson--Wright argument used above, applied with \(T=[p]\), gives the following:
for every \(u>0\), on the event \(\|\widehat\Sigma^{-1}\|_{\rm op}\lesssim 1\),
\[
    \mathbb P\left(
        \|\widehat \beta_{\rm OLS}-\beta\|_2^2
        >
        C\sigma^2 \frac{p+u}{N}
        \,\middle|\, X
    \right)
    \leq 2e^{-cu}.
\]
Moreover, under the sub-Gaussian design assumptions and \(N\gtrsim p\),
\[
    \mathbb P\left(\|\widehat\Sigma^{-1}\|_{\rm op}\lesssim 1\right)
    \geq 1-2e^{-cN}.
\]
Thus, choosing for instance \(u\asymp N\), we obtain
\[
    \|\widehat \beta_{\rm OLS}-\beta\|_2^2
    \lesssim \sigma^2
\]
with probability at least \(1-e^{-cN}\), since \(p\lesssim N\) and \(\sigma=O(1)\).

Now let
\[
    b_0 := \Pi_K \widehat \beta_{\rm OLS},
\]
where \(\Pi_K\) denotes (approximate) Euclidean projection onto \(K\). Since \(\beta\in K\) and the (approximate) Euclidean projection onto a closed convex set is (approximately) non-expansive (see Corollary A.2. in \cite{neykov2025polynomial}),
\[
    \|b_0-\beta\|_2
    =
    \|\Pi_K\widehat \beta_{\rm OLS}-\Pi_K\beta\|_2
    \lesssim
    \|\widehat \beta_{\rm OLS}-\beta\|_2.
\]
Consequently, with probability at least \(1-e^{-cN}\),
\[
    \|b_0-\beta\|_2^2 \lesssim \sigma^2 .
\]

Consider the shifted and rescaled observations
\[
    \widetilde Y_i
    :=
    \frac{Y_i-X_i^\top b_0}{2}
    =
    X_i^\top \theta + \widetilde \xi_i,
    \qquad
    \theta := \frac{\beta-b_0}{2},
    \qquad
    \widetilde \xi_i := \frac{\xi_i}{2}.
\]
Since \(b_0,\beta\in K\), symmetry and convexity of \(K\) imply
\[
    \theta=\frac{\beta-b_0}{2}\in K.\footnote{Although $\theta$ actually depends on $\xi$, inspection of the proof of Theorem \ref{linear:regression:theorem} shows that this is not an issue.}
\]
Furthermore, on the high-probability event above,
\[
    \|\theta\|_2^2
    \lesssim \sigma^2 .
\]
Hence, on this event, the shifted parameter belongs to
\[
    K_\rho := K\cap \rho B_2,
    \qquad
    \rho^2 \asymp \sigma^2,
\]
which has bounded Euclidean diameter. The set \(K_\rho\) is again the unit ball of
a symmetric norm, because
\[
    K_\rho
    =
    \left\{
        x:
        \max\left(\|x\|_K,\frac{\|x\|_2}{\rho}\right)\leq 1
    \right\},
\]
and the norm \(x\mapsto \max(\|x\|_K,\|x\|_2/\rho)\) is sign- and
permutation-invariant.

We may therefore apply the estimator developed above to the transformed regression
problem
\[
    \widetilde Y_i = X_i^\top \theta+\widetilde \xi_i
\]
over the bounded-diameter symmetric norm ball \(K_\rho\). If \(\widehat\theta\) denotes
the resulting estimator, define
\[
    \widehat \beta
    :=
    \Pi_K\bigl(b_0+2\widehat\theta\bigr).
\]
Since \(\beta=b_0+2\theta\in K\), another use of non-expansiveness of projection gives
\[
    \|\widehat\beta-\beta\|_2^2
    \lesssim
    \|b_0+2\widehat\theta-(b_0+2\theta)\|_2^2
    =
    4\|\widehat\theta-\theta\|_2^2.
\]
Thus the bounded-diameter result applied to \(K_\rho\) transfers back to the original
parameter \(\beta\), up to a universal constant factor. The failure event contributes at
most
\[
    d^2 e^{-cN}
\]
to the risk, where \(d=\operatorname{diam}(K)\). Therefore, provided
\[
    d^2 \lesssim \frac{e^{cN}}{N},
\]
this contribution is at most of order \(1/N\), and hence can be absorbed into the
usual minimax-rate upper bound. In this sense, the bounded-diameter assumption is not substantially restrictive in the moderate-dimensional regime.

Moreover, this rate is near minimax optimal in the case that $d^2 \lesssim \frac{e^{cN}}{N}$. This follows since the lower bound of \cite{prasadan2025characterizingminimaxratenonparametric} does not need a condition on the diameter, and the local entropy of the set $K_\rho$ is clearly smaller than the entropy of the set $K$, as $K_\rho \subseteq K$.
\end{remark}

\section{Discussion}\label{discussion:section}

This paper proposes a computationally efficient estimator for the Gaussian
sequence model over any well-balanced symmetric norm ball.  The estimator is
based on a simple principle: sort the coordinates of the observation, project onto the corresponding sparse sections of the constraint set, and select the sparsity level using a complexity penalty.  Throughout the paper, we assume that the noise level $\sigma$ is known, since it enters the model-selection penalty used by the estimator.  The main result shows that this procedure is near minimax optimal, up to logarithmic factors, over the full class of well-balanced symmetric norm balls.  Thus the paper extends the classical thresholding intuition for $\ell_p$ balls to a substantially broader class of permutation- and sign-invariant constraints.

A useful feature of the construction is that it avoids the direct construction of entropy packings.  Although entropy numbers and local entropy play a central role in the analysis, the estimator itself is much simpler: it only requires sorting, sparse projection, and penalized model selection.  In this sense, the algorithm gives a concrete statistical realization of the entropy theory for symmetric norm balls.  The Edmunds--Netrusov characterization of entropy numbers is the key geometric input, as it connects the minimax behavior of the problem to sparse approximation.

We also considered an extension of the method to random-design linear
regression in a moderate-dimensional setting.  It would be interesting to
understand whether this regression procedure can be extended to genuinely
high-dimensional covariate settings, where the dimension may be much larger
than the sample size. In fact from \cite{prasadan2025characterizingminimaxratenonparametric} we know that even in such high-dimensional situations the minimax rate is still determined through the entropy equation. Such an extension would likely require additional structural assumptions on the design, or a more delicate analysis of the interaction between the empirical covariance matrix and the symmetric norm constraint.  In particular, one would need to understand whether the sparse projection and model-selection ideas used here remain stable when the design does not preserve Euclidean geometry uniformly over all relevant sparse sections.

Another natural question is whether the procedure can be made adaptive to the unknown noise level.  The present algorithm assumes knowledge of $\sigma$, primarily through the penalty used to select the sparsity level.  It would be useful to develop a version of the method in which this quantity is estimated from the data, or in which the model-selection step is calibrated in a scale-adaptive way, while preserving the near-minimax guarantees obtained here.

The approach developed here is also conceptually quite different from the
method in \cite{neykov2025polynomial} for polynomial-time near-optimal
estimation over certain type-2 convex bodies.  The latter relies on a
geometric-functional-analytic property of the body, namely type-2 behavior,
and uses this structure to relate entropy numbers to more algorithmically
accessible quantities.  In contrast, the present paper exploits symmetry of
the norm directly.  The signed-permutation invariance reduces the search over all sparse supports to a much simpler ordering problem, and the
Edmunds--Netrusov theorem then identifies the correct approximation scale.
Thus the two methods apply to rather different geometric regimes and should be viewed as complementary.

A natural direction for future work is to determine whether the present idea, or a suitable modification of it, can be applied to constraint sets that are not symmetric in the sense used here.  For general norm balls or convex bodies without signed-permutation invariance, sorting the coordinates of $Y$ no longer identifies the best sparse sections, and the rearrangement argument used in the proof breaks down.  Nevertheless, it is plausible that analogous procedures could be developed for other classes of structured constraints, provided one can identify a computationally tractable family of approximating submodels whose approximation error is tied to the entropy numbers of the set.

Finding such families beyond symmetric norm balls would broaden the scope of
fast near-optimal estimation and could help clarify which geometric features
are truly responsible for computational tractability.

\bibliographystyle{abbrv}
\bibliography{polytime}

\newpage 

\appendix

\section{Supplemental Proofs}
\begin{proof}[Proof of Lemma \ref{simple:calculation}]
Set
\begin{align*}
    L:=\log(1+n/k),
    \qquad
    a:=\frac{k}{L},
\end{align*}
so that \(s=\lceil a\rceil\). First observe that, without the ceiling,
\begin{align*}
    a\log\left(\frac{en}{a}\right)\asymp k.
\end{align*}
Indeed, writing \(t:=n/k\ge 1\), we have
\begin{align*}
    a=\frac{k}{\log(1+t)}
\end{align*}
and therefore
\begin{align*}
    a\log\left(\frac{en}{a}\right)
    =
    \frac{k}{\log(1+t)}
    \log\left(e t \log(1+t)\right).
\end{align*}
Since
\begin{align*}
    \log\left(e t \log(1+t)\right)
    \asymp
    \log(1+t),
    \qquad t\ge 1,
\end{align*}
it follows that
\begin{align*}
    a\log\left(\frac{en}{a}\right)\asymp k.
\end{align*}

We now account for the ceiling. Let
\begin{align*}
    f(x):=x\log\left(\frac{en}{x}\right).
\end{align*}
There are two cases.

First suppose \(a<1\). Then \(s=\lceil a\rceil=1\), and hence
\begin{align*}
    s\log\left(\frac{en}{s}\right)=\log(en).
\end{align*}
On the other hand, \(a<1\) means \(k<L\), and since
\begin{align*}
    L=\log(1+n/k)\le \log(en),
\end{align*}
we have \(k\le \log(en)\). Therefore
\begin{align*}
    k+\log(en)\asymp \log(en)
    =
    s\log\left(\frac{en}{s}\right).
\end{align*}

Now suppose \(a\ge 1\). Since \(s=\lceil a\rceil\), we have
\begin{align*}
    a\le s\le a+1\le 2a.
\end{align*}
If \(a+1\le n\), then \(f\) is increasing on \([a,a+1]\), because
\begin{align*}
    f'(x)=\log(n/x)\ge 0
    \qquad \text{for } x\le n.
\end{align*}
Thus
\begin{align*}
    f(s)\ge f(a).
\end{align*}
Also, using \(s\le 2a\) and \(s\ge a\),
\begin{align*}
    f(s)
    =
    s\log\left(\frac{en}{s}\right)
    \le
    2a\log\left(\frac{en}{a}\right)
    =
    2f(a).
\end{align*}
Hence
\begin{align*}
    f(s)\asymp f(a)\asymp k.
\end{align*}

If instead \(a+1>n\), then \(a\asymp n\). Since also \(s=\lceil a\rceil\asymp n\), we have
\begin{align*}
    f(s)\asymp n
    \qquad\text{and}\qquad
    f(a)\asymp n.
\end{align*}
Therefore again
\begin{align*}
    f(s)\asymp f(a)\asymp k.
\end{align*}

Thus, in the case \(a\ge 1\),
\begin{align*}
    s\log\left(\frac{en}{s}\right)\asymp k.
\end{align*}
Moreover \(a\ge 1\) implies \(k\ge L=\log(1+n/k)\). Hence
\begin{align*}
    \log(en)
    =
    1+\log k+\log(n/k)
    \lesssim k,
\end{align*}
so
\begin{align*}
    k+\log(en)\asymp k.
\end{align*}
Combining this with the previous case gives the uniform bound
\begin{align*}
    s\log\left(\frac{en}{s}\right)
    \asymp
    k+\log(en).
\end{align*}

The final claim follows immediately: if \(k\gtrsim \log(en)\), then
\begin{align*}
    k+\log(en)\asymp k,
\end{align*}
and therefore
\begin{align*}
    s\log\left(\frac{en}{s}\right)\asymp k.
\end{align*}
\end{proof}

\begin{proof}[Proof of Lemma \ref{well-balancedness-norm}]
By permutation invariance,
\[
\|e_i\|_K = \|e_1\|_K = a
\qquad \text{for all } i=1,\dots,n.
\]

We first prove the upper bound. By the triangle inequality and homogeneity,
\[
\|x\|_K
=
\left\|\sum_{i=1}^n x_i e_i\right\|_K
\le
\sum_{i=1}^n |x_i| \|e_i\|_K
=
a\|x\|_1 .
\]
Since $\|x\|_1 \le \sqrt n \|x\|_2$, we obtain
\[
\|x\|_K \le a\sqrt n \|x\|_2 .
\]

For the lower bound, we use the coordinatewise monotonicity of symmetric 
norms: if $|u_i|\le |v_i|$ for all $i$, then
\[
\|u\|_K \le \|v\|_K .
\]
Let $j\in \arg\max_i |x_i|$. Then $\|x\|_\infty = |x_j|$. The vector
$|x_j|e_j$ is coordinatewise dominated by $x$ in absolute value, since
\[
\bigl||x_j| (e_j)_i\bigr| \le |x_i|
\qquad \text{for all } i.
\]
Therefore, by coordinatewise monotonicity,
\[
\||x_j|e_j\|_K \le \|x\|_K .
\]
Using homogeneity and permutation invariance,
\[
\||x_j|e_j\|_K
=
|x_j|\|e_j\|_K
=
a\|x\|_\infty .
\]
Hence
\[
a\|x\|_\infty \le \|x\|_K .
\]
Finally, since $\|x\|_\infty \ge \|x\|_2/\sqrt n$, we get
\[
\|x\|_K
\ge
a\|x\|_\infty
\ge
\frac{a}{\sqrt n}\|x\|_2 .
\]

Combining the two estimates gives
\[
\frac{a}{\sqrt n}\|x\|_2
\le
\|x\|_K
\le
a\sqrt n \|x\|_2 .
\]
The claimed inclusions follow directly from these norm inequalities.
\end{proof}

\end{document}